\documentclass{amsart}
\usepackage[english]{babel}
\usepackage{color,xspace}
\usepackage[normalem]{ulem}
\usepackage{mathrsfs}
\usepackage{stmaryrd}
\usepackage{amssymb}
\usepackage{pgfplots}
\usepackage{pgfplotstable}

\usepackage{caption}
\usepackage{subcaption}
\usepackage{csvsimple}
\usepackage{siunitx} 
\newcommand\RE{\mathbb{R}}
\newcommand\CO{\mathbb{C}}
\newcommand\dual[3]{\langle#1,#2\rangle_{#3}}
\newcommand\sfzero{\mathsf{0}}
\newcommand\sfA{\mathsf{A}}
\newcommand\sfB{\mathsf{B}}

\newcommand\sfM{\mathsf{M}}

\newcommand\sfx{\mathsf{x}}
\newcommand\sfy{\mathsf{y}}

\newcommand\Tf{T_{F}}
\newcommand\Tfh{T_{F,h}}

\newcommand\bfsigma{\boldsymbol{\sigma}}
\newcommand\bftau{\boldsymbol{\tau}}

\newcommand\bfH{\mathbf{H}}

\newcommand\bfn{\mathbf{n}}
\newcommand\grad{\operatorname\nabla}
\renewcommand\div{\operatorname{\mathrm{div}}}

\newcommand\Hdiv{\mathbf{H}(\div;\Omega)}

\newcommand{\vertiii}[1]{{\left\vert\kern-0.25ex\left\vert\kern-0.25ex\left\vert #1 
    \right\vert\kern-0.25ex\right\vert\kern-0.25ex\right\vert}}

\newcommand\gap{\hat\delta}
\newcommand\ggap{\delta}
\renewcommand\H{\mathcal{H}}
\newcommand\bfu{\mathbf{u}}

\newtheorem{theorem}{Theorem}
\newtheorem{lemma}[theorem]{Lemma}
\newtheorem{proposition}[theorem]{Proposition}
\newtheorem{corollary}[theorem]{Corollary}
\theoremstyle{remark}
\newtheorem{remark}{Remark}

\begin{document}

\title[DPG approximation of eigenvalue problems]{Discontinuous Petrov--Galerkin approximation of eigenvalue problems}

\author{Fleurianne Bertrand}
\address{Humboldt-Universit\"at zu Berlin, Germany}
\curraddr{}
\email{}
\thanks{}
\author{Daniele Boffi}
\address{King Abdullah University of Science and Technology, Saudi Arabia, and
University of Pavia, Italy}
\curraddr{}
\email{}
\thanks{}

\author{Henrik Schneider}
\address{Humboldt-Universit\"at zu Berlin, Germany}
\curraddr{}
\email{}
\thanks{}

\date{}

\begin{abstract}
In this paper, the discontinuous Petrov--Galerkin approximation of
the Laplace eigenvalue problem is discussed. We consider in particular 
the primal and ultra weak formulations of the problem and prove the convergence
together with a priori error estimates.
Moreover, we propose two possible error estimators and perform the
corresponding a posteriori error analysis. The theoretical results are confirmed 
numerically and it is shown that the error estimators can be used to design an 
optimally convergent adaptive scheme.
\end{abstract}

\maketitle

\section{Introduction}

DPG approximation of partial differential equations is a popular and effective
technique which has reached a quite mature level of discussion within the
scientific community.
The literature about DPG is pretty rich.
The method has been introduced in a series of
papers~\cite{DPG1,DPG2,DPG3,DPG4} during the last decade. Originally, the
method has been presented as a technique to design an intrinsically stabilized
scheme for advective problems. The main idea is to use a suitable
discontinuous trial and test functions that are tailored for stability. The
\emph{ideal} DPG formulation is turned into a \emph{practical} DPG
formulation~\cite{practical} where the test function space is easily computable
and arbitrary close to the optimal one. The DPG formulation comes with a
natural a posteriori error indicator that can be used for driving a robust
$hp$ adaptivity. Moreover, it has been shown that for one dimensional problems
the DPG method can be tuned to provide no phase errors in the case of
time-harmonic wave propagation.

After these pioneer works, several studies have been performed, showing that
the method can be applied to a variety of other problems. In particular: a
solid analysis has been presented in the case of the Laplace
equations~\cite{DPGPoisson}; the method has been proved locking free for
linear elasticity~\cite{elasticity}; it has been applied to Friedrichs-like
systems~\cite{Friedrichs}, including convection-diffusion-reaction, linear
continuum mechanics, time-domain acoustic, and a version of Maxwell's
equations; it has been applied to the Reissner--Mindlin plate bending
model~\cite{RM}, to the Helmholtz equation~\cite{Helmholtz}, to the Stokes
problem~\cite{Stokes}, to compressible flows~\cite{compressible}, to the
Navier--Stokes equations~\cite{NS}, and to the Maxwell
equation~\cite{Maxwell}.

In this paper, we are interested in the approximation of the Laplace eigenvalue
problem. We study the so-called \emph{primal} and \emph{ultra weak}
formulation of the Poisson problem: we refer the interested reader
to~\cite{DPGPrimal,DPGPoisson,practical} for the analysis and a discussion
of the primal and ultra weak formulations for the Laplace problem in the case
of the source problem. We will also look at a posteriori error estimators,
perform an a posteriori analysis, and show numerically the optimal convergence
of an adaptive scheme.  

Following what was recently done for the Least-Squares finite element
method~\cite{ls,lselas}, we study how the DPG method can approximate
eigenvalue problems.

After recalling the abstract setting for DPG approximations and the
convergence of eigenvalues and eigenfunctions in Section~\ref{se:pbsetting},
we apply the theoretical framework to the Laplace eigenvalue problem in
Section~\ref{se:laplace}, where we show the a priori estimates for the primal
and ultra weak formulations.
The a posteriori analysis is developed in Section~\ref{se:post} where we
consider a \emph{natural} error estimator related to what was studied
in~\cite{cdg} and the alternative error estimator introduced
in~\cite{Hellwig2019Adaptive} which is based on a suitable characterization of
the eigenfunctions in terms of Crouzeix--Raviart elements valid for the lowest
order approximation.
Finally, in Section~\ref{se:num} we present some numerical tests where the
theory is confirmed and where it is shown that an adaptive scheme for the DPG
approximation of the eigensolutions of the Laplace problem is optimally
convergent.

\section{Problem formulation and a priori analysis framework}
\label{se:pbsetting}

We start by describing the general structure of a DPG \emph{source} problem,
since this is useful in order to introduce our notation before dealing with
the corresponding eigenvalue problem.

The source problem we are studying is: find $u\in U$ such that
\begin{equation}
b(u,v)=\ell(v)\quad\forall v\in V,
\label{eq:DPGsource}
\end{equation}
where $U$ is a trial Hilbert space and $V$ is a test Hilbert space.
The bilinear form $b:U\times V\to\CO$ satisfies the assumptions
\begin{align}
&b(u,v)=0\quad\forall v\in V\Longrightarrow u=0 \label{eq:uniqueness}\\
& 0<C_1  := \inf_{u \in U\setminus \{ 0\}} \sup_{v\in V\setminus \{ 0\}}
\frac{|b(w,v)|}{\Vert w\Vert_U \Vert v \Vert_V} 
\quad\forall v\in V \label{eq:inf-sup-continuity}
\end{align}
$\ell:V\to\CO$ is a linear form.

The DPG formulation of~\eqref{eq:DPGsource} is obtained by introducing a
discrete space $U_h\subset U$ and an \emph{optimal} test function space given
by $V_{opt}=T(U_h)$, where $T:U\to V$ is the trial-to-test operator defined
as: find $Tu\in V$ such that
\begin{equation}
(Tu,v)_V=b(u,v)\quad\forall v\in V.
\label{eq:trial-to-test}
\end{equation}
The discrete \emph{source} problem is: find $u_h\in U_h$ such that
\begin{equation}
b(u_h,v)=\ell(v)\quad\forall v\in V_{opt}.
\label{eq:DPGsourceh}
\end{equation}
This is an ideal setting, since the actual computation of the test function
space $V_{opt}$ is not feasible in most applications. In general a
\emph{practical} DPG method is adopted where the test space $V_{opt}$ is
replaced by $V_{opt,h}$, obtained after introducing a finite dimensional
subspace $V_h\subset V$ and defining a discrete trial-to-test operator $T_h$
as in~\eqref{eq:trial-to-test} with $V$ replaced by $V_h$. Then the practical
optimal test space is $V_{opt,h}=T_h(U_h)$.
If $V_h$ is a space of discontinuous piecewise
polynomials, then the computation of the test functions is cheap, involving
the solution of a block diagonal system.

We assume that there exists a linear operator
$\Pi:V\to V_h$ and $C_\Pi$ such that for all $u_h\in U_h$ and all $v\in V$
\begin{align}
    &b(u_h,v-\Pi v)=0\\
    &\Vert\Pi v\Vert_V \leq C_\Pi \Vert v \Vert_V. 
    \label{eq:operatorPI}
\end{align}
A fundamental characterization of the solution $u_h$ of the DPG system is
given by the following mixed problem that is defined only via the discrete
spaces $U_h$ and $V_h$ without the need of the trial-to-test operator $T$:
find $u_h\in U_h$ and $\varepsilon_h\in V_h$ such that
\begin{equation}
\left\{
\aligned
&(\varepsilon_h,v_h)_V+b(u_h,v_h)=\ell(v_h)&&\forall v_h\in V_h\\
&\overline{b(z_h,\varepsilon_h)}=0&&\forall z_h\in U_h.
\endaligned
\right.
\label{eq:DPGmixedh}
\end{equation}

We are now ready to introduce the eigenvalue problem associated
with~\eqref{eq:DPGsource} and its approximation corresponding
to~\eqref{eq:DPGmixedh}.

Usually the space $U$ consists of two components and can be presented as
$U=U_0\times U_1$, where $U_0$ is a functional space defined on $\Omega$
(volumetric part) and $U_1$ is the remaining part that can be defined on
$\Omega$ or on the skeleton of a given triangulation.
Let $\H$ be a Hilbert pivot space so that we have the usual triplet
$U_0\subset\H\simeq\H'\subset U_0'$ and consider a bilinear form
$m:\H\times V\to\CO$.
The continuous eigenvalue problem is: find eigenvalues $\lambda\in\CO$ and
eigenfunctions $u=(u_0,u_1)\in U=U_0\times U_1$ with $u_0\ne0$ such that
\begin{equation}
b(u,v)=\lambda m(u_0,v)\quad\forall v\in V.
\label{eq:DPGeig}
\end{equation}
In order to state the appropriate (compactness) assumptions, we introduce the
solution operator $\Tf:\H\to\H$ as follows: $\Tf f\in\H$ is the component
$u_0$ of the solution $u\in U$ to
\begin{equation}
b(u,v)=m(f,v)\quad\forall v\in V.
\label{eq:defT}
\end{equation}
We assume that~\eqref{eq:defT} is uniquely solvable and the operator $\Tf$ is
compact.

The discrete space $U_h$ is analogously made of two components
$U_{0,h}\subset U_0$ and $U_{1,h}\subset U_1$. The DPG discretization,
corresponding to the mixed formulation~\eqref{eq:DPGmixedh}, is given by: find
$\lambda_h\in\CO$ such that for some
$u_h=(u_{0,h},u_{1,h})\in U_h=U_{0,h}\times U_{1,h}$ with $u_{0,h}\ne0$ and
some $\varepsilon_h\in V_h$ it holds
\begin{equation}
\left\{
\aligned
&(\varepsilon_h,v_h)_V+b(u_h,v_h)=\lambda_hm(u_{0,h},v_h)&&\forall v_h\in V_h\\
&\overline{b(z_h,\varepsilon_h)}=0&&\forall z_h\in U_h.
\endaligned
\right.
\label{eq:DPGeigh}
\end{equation}

The matrix form of this formulation is given by
\begin{equation}
\left(
\begin{matrix}
\sfA&\sfB^\top\\
\sfB&\sfzero
\end{matrix}
\right)
\left(
\begin{matrix}
\sfx\\\sfy
\end{matrix}
\right)=
\lambda\left(
\begin{matrix}
\sfzero&\sfM\\
\sfzero&\sfzero
\end{matrix}
\right)
\left(
\begin{matrix}
\sfx\\\sfy
\end{matrix}
\right)
\label{eq:matrix}
\end{equation}
where $\sfx$ is the vector representation of $\varepsilon_h\in V_h$ and $\sfy$
of $u_h\in U_h$.

We introduce the discrete counterpart $\Tfh:\H\to\H$ of the solution operator
$\Tf$ as follows: $\Tfh f\in U_{0,h}\subset\H$ is the component
$u_{0,h}\in U_{0,h}$ of the solution $u_h\in U$ of the following problem, for
some $\varepsilon_h\in V_h$,
\begin{equation}
\left\{
\aligned
&(\varepsilon_h,v_h)_V+b(u_h,v_h)=m(f,v_h)&&\forall v_h\in V_h\\
&\overline{b(z_h,\varepsilon_h)}=0&&\forall z_h\in U_h.
\endaligned
\right.
\label{eq:defTh}
\end{equation}

In order show a priori estimates for the eigensolution computed with the DPG
method we are going to use the classical Bab\v uska--Osborn
theory~\cite{BaOs,acta}. Let us denote by $\lambda_i$, $i=1,\dots$, the
eigenvalues of the continuous problem~\eqref{eq:DPGeig} sorted such that
\[
0<|\lambda_1|\le|\lambda_2|\le\cdots
\]
and by $E_i=\mathrm{span}(u_{0,i})$, $i=1,\dots$, the corresponding
eigenspaces. In case of multiple eigenvalues we repeat them so that each $E_i$
is one dimensional. Analogous notation $\lambda_{i,h}$ and $E_{i,h}$,
$i=1,\dots$, is adopted for the discrete problem~\eqref{eq:DPGeigh}.

\begin{theorem} \label{th:babuska_osborn}
If the following convergence in norm holds true
\begin{equation}
\|(\Tf-\Tfh) f\|_\H\le\rho(h)\|f\|_\H\quad\forall f\in\H
\label{eq:cunif}
\end{equation}
with $\rho(h)$ tending to zero as $h$ goes to zero, then the discrete
eigenvalues and eigenfunctions converge to the continuous ones.
That is, any compact set $K$ included in the resolvent set of $\Tf$ is
included in the resolvent set of $\Tfh$ for $h$ small enough (absence of
spurious modes); moreover, if $\lambda_i$ is an eigenvalue of algebraic
multiplicity $m$ then there are exactly $m$ discrete eigenvalues
$\lambda_{i_j,h}$, $j=1,\dots m$, tending to $\lambda_i$ (convergence).
\end{theorem}

In order to estimate the rate of convergence, as usual, we shall make use of
the gap between subspaces of Hilbert spaces defined as
\[
\gap(A,B)=\max\{\ggap(A,B),\ggap(B,A)\},
\]
where
\[
\ggap(A,B)=\sup_{\substack{a\in A\\\|a\|=1}}\ggap(a,B)
\]
with
\[
\ggap(a,B)=\inf_{b\in B}\|a-b\|_W
\]
for $A$ and $B$ closed subspaces of a Hilbert space $W$.
If $W=U_0$ and $E$ is the $m$ dimensional eigenspace of the continuous problem
corresponding to $\lambda_i$ (see the setting of
Theorem~\ref{th:babuska_osborn}), we introduce the following quantity related
to $E$
\[
\gamma_h=\|(\Tf-\Tfh)|_{E}\|_{\mathcal L(U_0)}
\]
and, if $E^*$ is the corresponding eigenspace of the adjoint operator $\Tf^*$,
we consider the following quantity
\[
\gamma^*_h=\|(\Tf^*-\Tfh^*)|_{E^*}\|_{\mathcal{L}(U_0')},
\]
where $\Tfh^*$ is the discrete solution operator associated with the adjoint
problem.
Then we recall the following classical result.
\begin{theorem}
Under the hypothesis of Theorem~\ref{th:babuska_osborn} it holds
\[
\gap(E_h,E)\le C\gamma_h,
\]
where $E_h$ is the direct sum of the eigenspaces of the $m$ eigenvalues
approximating $\lambda_i$. Moreover, if $\alpha$ is the ascent multiplicity of
$\lambda_i$, it holds
\[
\max_{j=1,\dots,m}|\lambda_i-\lambda_{i_j,h}|\le
C(\gamma_h\gamma_h^*)^{1/\alpha}.
\]
\label{th:babuska_osborn2}
\end{theorem}

\section{The Laplace eigenvalue problem}
\label{se:laplace}

Let $\Omega\subset\mathbb R^2$ be a bounded polygonal domain.
We are interested in the standard Dirichlet eigenvalue problem for the Poisson
equation: find $\lambda$ such that for a nonzero $u$ we have
\[
\left\{
\aligned
-\Delta u &= \lambda u&&\text{in }\Omega\\
u&=0&&\text{on }\partial\Omega.
\endaligned
\right.
\]
We look at a conforming triangulation $\Omega_h$ and its skeleton
$\partial\Omega_h$.

We will use two different formulations; namely, the so
called \emph{primal} and \emph{ultra-weak} formulations (see, for
instance,~\cite{DPGPrimal,DPGPoisson,practical}).
\subsection{Primal formulation}
\label{se:primal}
The formulation we are considering has been presented in~\cite{DPGPrimal} and
fits within our setting with the following choices
\[
\aligned
&U=U_0\times U_1\\
&U_0=H^1_0(\Omega)\\
&U_1=H^{-1/2}(\partial\Omega_h)\\
&\H=H^1(\Omega)\\
&V=H^1(\Omega_h)\\
&b(u,\hat\sigma_n;v)=
(\grad u,\grad v)_{\Omega_h}-\dual{\hat\sigma_n}{v}{\partial\Omega_h}\\
&m(u,v)=(u,v)_{\Omega_h},
\endaligned
\]
where as usual the symbol $\dual{\cdot}{\cdot}{\partial\Omega_h}$ denotes the
action of a functional in $H^{-1/2}(\partial\Omega_h)$ and
$(\cdot,\cdot)_{\Omega_h}$ the broken $L^2$ scalar product. We recall the
definition of $H^{-1/2}(\partial\Omega_h)$ as
\[
\left\{\tau\in\bigotimes_K H^{-1/2}(\partial K):
\tau|_{\partial K}=\bftau\cdot\bfn|_{\partial K}
\text{ for some }\bftau\in\Hdiv,\ \forall K\in\Omega_h\right\}.
\]

We make use of the following discrete spaces for any integer $k\geq 1$
\[
\aligned
    &U_{h,0} =  S^k_0(\Omega_h) \\
    &U_{h,1} =  P_{k-1}(\partial \Omega_h) \cap U_1\\
    &U_h = U_{h,0} \times U_{h,1} \\
    &V_h = P_{k+1}(\Omega_h).
\endaligned
\]
Here we denote 
\begin{align*}
    &S^k_0 (\Omega_h) := P_k(\Omega_h) \cap C(\bar\Omega). 
\end{align*}
\begin{remark}
In this section we are using the standard notation $u$ (resp.\ $u_h$) for the
volumetric part of the solution and $\hat\sigma$ (resp.\ $\hat\sigma_n$) for
the skeleton part. They correspond to $u_0$ (resp.\ $u_{0,h}$) and $u_1$
(resp.\ $u_{1,h}$) in the abstract presentation of the previous section.
Analogous notation will be used in the next section for the ultra-weak
formulation.
\end{remark}

The uniform convergence~\eqref{eq:cunif} is usually proved by employing some a
priori estimates of the source problem.
The standard estimate for the source from \cite{DPGPrimal} reads as follows 
\begin{equation}
\aligned
&\Vert u - u_h \Vert_{H^1(\Omega)} + \Vert \hat \sigma_n -\hat \sigma_{h,n} 
    \Vert_{H^{-1/2}(\partial \Omega_h)}\\
&\qquad
\le C\inf_{(w_h,\hat r_{h,n} )\in U_h} (\Vert u - w_h \Vert_{H^1(\Omega)} 
+ \Vert \hat \sigma_n -\hat r_{h,n} \Vert_{H^{-1/2}(\partial \Omega_h)}).
\endaligned
\label{eq:DPGprimalconvergence}
\end{equation}
Since we have chosen $\H=H^1(\Omega)$, the uniform convergence follows
from~\eqref{eq:DPGprimalconvergence} as it is shown in the next proposition.

\begin{proposition} \label{uniform}
Let $(u,\hat\sigma_n)\in U$ be the solution of the source
problem~\eqref{eq:DPGsource} with right hand side $f$ in $H^1(\Omega)$ and
assume that $u$ belongs to $H^{1+s}(\Omega)$ for some $s\in (1/2, k+1]$, where
$k$ is the order of the approximation introduced above.
Then the uniform convergence~\eqref{eq:cunif} holds true.
\end{proposition}
\begin{proof}
Let $(u_h, \hat\sigma_{h})\in U_h$ be the numerical solution corresponding to
our right hand side $f$ in $L^2(\Omega)$. Then the regularity assumptions
together with the natural error estimates recalled above (see
also~\cite[Eq.~(5.1)]{DPGPrimal} imply
\[
\|u-u_h\|_{H^1(\Omega)}+
\|\hat\sigma_n-\hat\sigma_{h,n}\|_{H^{-1/2}(\partial\Omega_h)}
\le C h^s\|f\|_{H^1(\Omega)}.
\]
Due to the definitions $\Tf(f)=u$ and $\Tfh(f)=u_h$ this implies
\[
\|(\Tf-\Tfh)f\|_{H^1(\Omega)}\le C h^s\|f\|_{H^1(\Omega)},
\]
which proves the uniform convergence~\eqref{eq:cunif} with $\H=H^1(\Omega)$
and $\rho(h)\le C h^s$.
\end{proof}
\begin{corollary}
Proposition~\ref{uniform} holds also for the following discrete spaces for any
odd integer $k\ge 1$
\[
\aligned
&U_{h,0}=S^k_0(\Omega_h)\
&U_{h,1}=P_{k-1}(\partial\Omega_h)\cap U_1\\
&\tilde V_h=P_{k}(\Omega_h).
\endaligned
\]
\end{corollary}
\begin{proof}
The result follows from \cite[Theorem 3.5]{DPGReduced} with the same arguments as in 
Proposition \ref{uniform}.
\end{proof}
We are now in a position to state our main conclusion of this section. For
readiness, we consider the case of a simple eigenvalue. Natural modifications
apply in case of higher multiplicity.
\begin{theorem}
Let us consider the DPG primal approximation of the Laplace eigenvalue problem
as discussed in Proposition~\ref{uniform}. Then the conclusions of
Theorems~\ref{th:babuska_osborn} and~\ref{th:babuska_osborn2} hold true. 
In particular, let $\lambda$ be a simple eigenvalue of the continuous problem
corresponding to an eigenspace $E$ belonging to $H^{1+s}(\Omega)$ and let
$\lambda_h$ be the approximation of $\lambda$ with discrete eigenspace $E_h$.
Then we have
\begin{equation}
\aligned
&\gap(E,E_h)\le Ch^{\tau}\\
&|\lambda-\lambda_h|\le Ch^{2\tau}
\endaligned
\label{eq:th1}
\end{equation}
with $\tau := \min \{s, k\}$.
\label{th:th1}
\end{theorem}
\begin{proof}
Theorem~\ref{th:babuska_osborn} follows from the convergence in
norm~\eqref{eq:cunif} which we proved in Proposition~\ref{uniform}.

In order to verify the rates of convergence shown in~\eqref{eq:th1} we have to
compute the quantities $\gamma_h$ and $\gamma^*_h$ related to the convergence
of the DPG primal formulation and of its adjoint formulation, respectively.
For the primal formulation, we can estimate $\gamma_h$ by using the optimal
bound~\eqref{eq:DPGprimalconvergence}, thus obtaining
\[
\gamma_h=O(h^\tau)
\]
which gives the first in~\eqref{eq:th1}.

The adjoint problem corresponding to the primal formulation has been
considered extensively in~\cite{DPGReduced} for the proof of a duality
argument. In our notation, the continuous formulation of the adjoint problem
corresponding to~\eqref{eq:DPGmixedh} (see~\cite[Eq.~(20)]{DPGReduced}) is:
given $g\in\H$, find $\varepsilon^*\in V$ and
$\bfu^*=(u^*,\hat\sigma_n^*)\in U$ such that
\begin{equation}
\left\{
\aligned
&(\varepsilon^*,w)_V+(\grad u^*,\grad w)_{\Omega_h}
-\dual{\hat\sigma_n^*}{w}{\partial\Omega_h}=0&&\forall w\in V\\
&(\grad\varepsilon^*,\grad v)_{\Omega_h}=(g,v)_{\Omega_h}&&\forall v\in U_0\\
&\dual{\hat\tau_n}{\varepsilon^*}{\partial\Omega_h}=0&&
\forall\hat\tau_n\in U_1
\endaligned
\right.
\label{eq:dual}
\end{equation}
In order to estimate $\gamma^*_h$ we need to consider the discretization
of~\eqref{eq:dual}. It is apparent that the left hand side of the adjoint
problem is the same as the one corresponding the standard primal formulation,
so that, denoting by $(\varepsilon^*_h,u^*_h,\hat\sigma^*_{n,h})\in V_h\times
U_{h,0}\times U_{h,1}$ the discrete solution, the following quasi-optimal a
priori estimate holds true
\begin{equation}
\aligned
&\|\varepsilon^*-\varepsilon^*_h\|_V+\|u^*-u^*_h\|_{U_0}+
\|\hat\sigma_n^*-\hat\sigma^*_{n,h}\|_{U_1}\\
&\qquad\le\inf_{(\delta,v,\hat\tau_n)\in V_h\times U_{h,0}\times U_{h,1}}
(\|\varepsilon^*-\delta\|_V+\|u^*-v\|_{U_0}
+\|\hat\sigma_n^*-\hat\tau_n\|_{U_1})
\endaligned
\label{eq:ratedual}
\end{equation}
It remains to check the regularity of the solution of the dual
problem~\eqref{eq:dual}. In~\cite[Eq.~(21)]{DPGReduced} it is observed that the strong
form of the dual problem is as follows
\[
\aligned
&-\Delta\varepsilon^*=g&&\text{in }\Omega\\
&\varepsilon^*=0&&\text{on }\partial\Omega\\
&\Delta u^*=\varepsilon^*+g&&\text{in }\Omega\\
&u^*=0&&\text{on }\partial\Omega\\
&\hat\sigma_n^*=\grad(\varepsilon^*+u^*)\cdot\bfn&&\text{on }\partial K\
\forall K\in\Omega_h
\endaligned
\]
In particular, since the dual problem involves the Laplace operator, the
regularity of its solution will be related to the same Sobolev exponent $1+s$
as for the original Laplace problem.
This implies that from the rate of convergence predicted
by~\eqref{eq:ratedual} we can obtain
\[
\gamma^*_h=O(h^\tau)
\]
Hence, the double order of convergence for the eigenvalues is proved, which
concludes our proof.
\end{proof}
Theorem~\ref{th:th1} estimates the eigenfunction error in the energy norm of
$U_0$. It is interesting to observe that the when the error is estimated in
$L^2(\Omega)$ then higher order can be achieved. This is stated in the
following proposition.
\begin{proposition}
Under the same assumptions and notation of Theorem~\ref{th:th1}, the
following estimate holds true
\begin{equation}
\gap_0(E,E_h)\le Ch^{\upsilon+\tau}
\label{eq:pr1}
\end{equation}
with $\upsilon=\min\{s,1\}$, where $\gap_0$ denotes the gap in the
$L^2(\Omega)$ norm.
\label{pr:pr1}
\end{proposition}
\begin{proof}
The result is a consequence of the analogous one valid for the source problem
which has been proved in~\cite[Thm.~3.1]{DPGReduced} in the case of a convex
domain. This is a standard Aubin--Nitche duality argument. The general case
follows by inspecting the proof of Theorem~3.1 dealing with Assumption~2.9
where the regularity of the dual (adjoint) problem is discussed.
Reference~\cite{DPGReduced} studies the case of full regularity $s=1$; when
$s<1$ exactly the same arguments give the result with $C_3(h)=h^s$.
\end{proof}
We now switch back momentarily to the notation of the abstract setting, where
the component $u$ of the solution was denoted by $u_0$ and the component
$\hat\sigma$ by $u_1$. Then we observe that the convergence estimates in the
first of~\eqref{eq:th1} and in~\eqref{eq:pr1} imply that given $u_0\in E$
there exists $u_{0,h}\in E_h$ such that
\begin{equation}
\aligned
&\|u_0-u_{0,h}\|_{U_0}\le Ch^\tau\|u_0\|_{H^{1+s}(\Omega)}\\
&\|u_0-u_{0,h}\|_{L^2(\Omega)}\le Ch^{\tau+\upsilon}\|u_0\|_{H^{1+s}(\Omega)}.
\endaligned
\label{eq:estprimal}
\end{equation}
When also the other component of the eigenfunction is considered, we have the
following a priori estimate
\begin{equation}
\|u-u_h\|_U\le Ch^\tau\|u_0\|_{H^{1+s}(\Omega)}
\label{eq:estprimalfull}
\end{equation}
where the components $u_1$ and $u_{1,h}$ of $u$ and $u_h$ are the ones
corresponding to $u_0$ and $u_{0,h}$ in~\eqref{eq:DPGeig}
and~\eqref{eq:DPGeigh}, respectively.
\subsection{Ultra weak formulation}  
\label{se:UW}
The DPG ultra weak formulation for the Laplace eigenvalue problem fits within
our abstract setting with the following choices
(see~\cite{DPGPoisson,practical} for more details):
\[
\aligned
&U=U_0\times U_1\\
&U_0=L^2(\Omega)\\
&U_1=L^2(\Omega)^2\times H^{1/2}_0(\partial\Omega_h)\times H^{-1/2}(\partial\Omega_h)\\
&\H=U_0\\
&V=H^1(\Omega_h)\times\bfH(\div;\Omega_h)\\
&b(u,\bfsigma,\hat u,\hat\sigma_n;v,\bftau)=
(\bfsigma,\bftau)_{\Omega_h}-(u,\div\bftau)_{\Omega_h} 
+ \dual{\hat u}{\bftau\cdot\bfn}{\partial\Omega_h}\\
&\hspace{3cm}-(\bfsigma,\grad v)_{\Omega_h} +\dual{v}{\hat\sigma_n}{\partial\Omega_h}\\
&m(u;v,\bftau)=(u,v)_{\Omega_h},
\endaligned
\]
where the space $H^{1/2}_0(\partial\Omega_h)$ is defined as
\[
\left\{\hat w\in\bigotimes_K H^{1/2}(\partial K):
\hat w|_{\partial K}=w|_{\partial K}\text{ for some }w\in H^1_0(\Omega),\
\forall K\in\Omega_h\right\}
\]
and $H^1(\Omega_h)$ and $\bfH(\div;\Omega_h)$ denote broken functional
spaces on the mesh $\Omega_h$.

We choose the following discrete spaces with $k\geq 0$
\begin{subequations}
    \label{ultradiskret}
\begin{align}
   U_h &:= P_k(\Omega_h) \times P_k(\Omega_h;\mathbb R^2) \times S^{k+1}_0(\partial\Omega_h) 
   \times P_k(\partial\Omega_h)\\
   V_h &:= P_{k+2}(\Omega_h) \times P_{k+2}(\Omega_h;\mathbb R^2)
\end{align}
\end{subequations}
where $S^{k+1}_0(\partial \Omega_h)$ is defined as follows 
\begin{align*}
    &S^{k+1}_0(\partial \Omega_h) := \gamma_0 (S^{k+1}_0(\Omega_h)\cap H^1_0(\Omega)).
\end{align*}
Here $\gamma_0$ denotes the canonical trace operator from $H^1(\Omega)$ to
$H^{1/2}(\partial\Omega_h)$.

Also in this case we can use the a priori estimates in order to show the
uniform convergence~\eqref{eq:cunif}.

\begin{proposition}\label{ultra-apriori}

Let $\Tf:\H\to\H$ be the solution operator associated with the continuous
problem and $\Tfh$ its discrete counterpart as defined in~\eqref{eq:defT}
and~\eqref{eq:defTh}, respectively. Assume that the solution $u$ of the
Poisson problem with $f$ in $L^2(\Omega)$ belongs to $H^{1+s}(\Omega)$ for
some $s\in (1/2 ,k+1]$, where $k$ is the order of approximation used
in~\eqref{ultradiskret}.  Then the convergence in norm~\eqref{eq:cunif} holds true.

\end{proposition}

\begin{proof}

The a priori error analysis presented, for instance,
in~\cite[Cor.\ 6]{superconv-ultra} reads
\[
\|u-u_h\|_{U_0}\le C h^s\|f\|_{L^2(\Omega)}
\]
%
which implies the uniform convergence~\eqref{eq:cunif}.
\end{proof}

The rate of convergence that follows naturally from the a priori error
estimate for the source problem is presented in the following theorem.

\begin{theorem}
Let us consider the DPG ultraweak approximation of the Laplace eigenvalue problem
as discussed in Proposition~\ref{ultra-apriori}. Then the conclusions of
Theorems~\ref{th:babuska_osborn} and~\ref{th:babuska_osborn2} hold true. In
particular, let $\lambda$ be a simple eigenvalue of the continuous problem
corresponding to an eigenspace $E$ belonging to $H^{1+s}(\Omega)$ and let
$\lambda_h$ be the approximation of $\lambda$ with discrete eigenspace $E_h$.
Then we have
\[
\aligned
&\gap(E,E_h)\le Ch^\tau\\
&|\lambda-\lambda_h|\le Ch^{2\tau}
\endaligned
\]
with $\tau=\min\{s,k+1\}$.
\label{th:th2}
\end{theorem}
\begin{proof}
We omit the technical detail of the proof that follows the same lines as the
proof of Theorem~\ref{th:th1}. In particular, the estimates are obtained by
inspecting the a priori estimates of the ultraweak formulation and of its
adjoint under our regularity assumptions.
The a priori estimates of the ultraweak formulation have been proved
in~\cite[Cor.~4.1]{DPGPoisson} and read as follows
\[
\aligned
\|u-u_h\|_{L^2(\Omega)}+\|\bfsigma-\bfsigma_h\|_{L^2(\Omega)}+
\|\hat u-\hat u_h\|_{H^{1/2}_0(\partial\Omega_h)}+
\|\hat\sigma_n-\hat\sigma_{n,h}\|_{H^{-1/2}(\partial\Omega_h)}\\
\le C
h^\tau(\|u\|_{H^{1+\tau}(\Omega)}+\|\bfsigma\|_{H^{1+\tau}(\Omega)})
\endaligned
\]
\end{proof}
It is also possible to introduce a slightly different lowest order
approximation for the ultra weak formulation. The discretization reads as
follows
\[
\aligned
&\tilde U_h= P_0(\Omega_h) \times P_0(\Omega_h;\RE^2)
\times S^1_0(\partial\Omega_h)\times P_0(\partial \Omega_h)\\
&\tilde V_h= P_1(\Omega_h)\times RT^{PW}_0(\Omega_h),
\endaligned
\]
where $RT^{PW}_0(\Omega_h)$ denotes the discontinuous Raviart--Thomas space of
lowest degree.

\begin{corollary}
    With the same assumptions as in Theorem~\ref{th:th2} 
    it holds for the lowest order case that
\[
\aligned
&\gap(E,E_h)\le Ch^\tau\\
&|\lambda-\lambda_h|\le Ch^{2\tau}
\endaligned
\]
\end{corollary}
\begin{proof}
As in the case of Theorems~\ref{th:th1} and~\ref{th:th2}, the result is
obtained by inspecting the a priori estimates of the ultraweak formulation and
of its adjoint under our regularity assumptions.
The a priori estimates for this choice of discrete spaces can be found
in~\cite[Theorem 3.3]{loworderultraweak}.
\end{proof}

For reasons that will become clearer in the next section, it will be useful to
have a higher-order estimate of the $U_0$ component of the solution in the
spirit of what we obtained in~\eqref{eq:estprimal} for the primal formulation.
This property can be achieved by augmenting the approximating space $U_{0,h}$
along the lines of what was proposed in~\cite{superconv-ultra}.

We then choose the following discrete spaces with $k\geq 0$
\begin{align*}
   U_h &:= P_{k+1}(\Omega_h) \times P_k(\Omega_h;\mathbb R^2) \times S^{k+1}_0(\partial\Omega_h)
   \times P_k(\partial\Omega_h)\\
   V_h &:= P_{k+2}(\Omega_h) \times P_{k+2}(\Omega_h;\mathbb R^2)
\end{align*}
where the order of the polynomials approximating $U_0$ is raised from $k$ to
$k+1$.

We now revert back to the notation of the abstract setting where the symbols
$u$ and $u_h$ refer to pairs $(u_0,u_1)$ and $(u_{0,h},u_{1,h})$ in $U$ and
$U_h$, respectively.

We use the improved a priori estimates obtained in~\cite[Thm.\ 10]{superconv-ultra}, which
reads 
\begin{align*}
\Vert u_0 - u_{0,h} \Vert_{U_0} \leq C h^{s + s'} \Vert f \Vert_{L^2(\Omega)},
\end{align*}
where $s' \in (1/2,1]$ denotes the regularity shift of an auxiliary problem
used for a duality argument. 
This implies now the following result.

\begin{theorem}

With the same assumptions as in Theorem~\ref{th:th2}, the augmented
formulation provides the following a priori error estimates. Given $u_0\in E$
and its corresponding $u=(u_0,u_1)\in U$, there exists $u_{0,h}\in E_h$ and
its corresponding $u_h=(u_{0,h},u_{1,h})\in U_h$ such that
\[
\aligned
&\|u-u_h\|_U\le Ch^\tau\|u_0\|_{H^{1+s}(\Omega)}\\
&\|u_0-u_{0,h}\|_{U_0}\le Ch^{\tau+s'}\|u_0\|_{H^{1+s}(\Omega)}
\endaligned
\]
with $\tau=\min\{s,k+1\}$ and $s'\in (1/2,1]$ is the
regularity shift defined in~\cite[Eq.\ (14)]{superconv-ultra}.
\label{th:UW+}
\end{theorem}

\section{A posteriori error analysis}
\label{se:post}

In this section we present and discuss two error estimators that can be used
in the framework of a posteriori analysis and adaptive schemes.

\subsection{The natural error estimator}

We start with the most natural error estimator, associated with the energy
residual, that has been considered in~\cite{cdg} for the source problem. The
residual is the component $\varepsilon_h$ of the solution to
problem~\eqref{eq:DPGmixedh}. In the setting of~\cite{cdg} we can consider two
operators $B,M:U\to V'$ defined as
\[
\aligned
&(Bu)(v) := b(u,v)&&\forall u\in U,\ \forall v\in V \\
&(Mu)(v) := m(u_0,v)&&\forall u\in U,\ \forall v\in V,
\endaligned
\]
where we are adopting the splitting $u=(u_0,u_1)\in U=U_0\times U_1$ as
considered in Section~\ref{se:pbsetting}.
We denote the operator norm as $C_B$ and $C_M$.
The natural error indicator studied theoretically in~\cite{cdg} is the
\emph{global} indicator
\begin{equation}
\eta=\|\varepsilon_h\| = \Vert \lambda_h Mu_h -Bu_h\Vert_{V'}
\label{eq:eta1}
\end{equation}
(see~\eqref{eq:DPGmixedh}), where the definition of $Mu_h$ makes use of the
component $u_{0,h}$ of $u_h$; the practical implementation of an adaptive
scheme based of $\eta$ employs a localized version of it.

With natural modifications of the analysis of~\cite{cdg} global efficiency and
reliability can be proved. Both properties rely on the following (usual)
higher order term $\lambda u_0-\lambda_h u_{0,h}$ (see Section~\ref{se:HOT}).

\begin{proposition}[Reliability and Efficiency]\label{r&r}
Let us examine an eigenvalue $\lambda$ of problem~\eqref{eq:DPGeig} of
multiplicity one with eigenfunction $u\in U$ and the corresponding discrete
eigenpair $(\lambda_h,u_h)$.
Assume that~\eqref{eq:uniqueness}, \eqref{eq:inf-sup-continuity}
and~\eqref{eq:operatorPI} are satisfied. Then the following reliability and
efficiency estimates hold true
\[
\aligned
&C_1\Vert u-u_h\Vert_U\le\sqrt{\eta^2+(\Vert\lambda(Mu)\circ(1-\Pi)\Vert_{V^*}
+\eta\Vert\Pi\Vert)^2}+C_M\Vert\lambda u_0-\lambda_h u_{0,h}\Vert\\
&\eta^2\le C_B^2\Vert u-u_h\Vert^2_U+C_M^2\Vert\lambda u_0-\lambda_h u_{0,h}\Vert^2.
\endaligned
\]
\end{proposition}
\begin{proof}
We define the error $e:=u-u_h$, the error representation function
$\varepsilon$ by
\begin{equation}
\dual{\varepsilon}{y}{V}=\lambda_h m(u_{0,h},v)-b(u_h,v)\quad\forall v\in V,
\label{eq:error-representationI}
\end{equation}
and its approximation $\varepsilon_h$ by
\begin{equation}
\dual{\varepsilon_h}{y_h}{V}=\lambda_h m(u_{0,h},v_h)-b(u_h,v_h)
\quad\forall v_h\in V_h.
\label{eq:error-representationII}
\end{equation}
From the inf-sup condition \eqref{eq:inf-sup-continuity} it follows
\[
C_1\Vert e\Vert_U\le\Vert Be\Vert_{V^*}\le
\Vert \varepsilon \Vert_V + \Vert M(\lambda u - \lambda_h u_h) \Vert_{V'} \le
\Vert\varepsilon\Vert_V+C_M\Vert\lambda u_0-\lambda_h u_{0,h}\Vert.
\]
From~\eqref{eq:error-representationI} and~\eqref{eq:error-representationII} we
see that
\[
\delta := \varepsilon - \varepsilon_h \perp V_h,
\]
so that Pythagoras theorem gives
\begin{equation}
\Vert \varepsilon \Vert^2_V = \Vert \varepsilon_h\Vert_V^2 + \Vert\delta\Vert^2_V.
\label{decomposition}
\end{equation}
Noting that $\Pi\delta\in V_h\perp\delta$ and from~\eqref{eq:operatorPI} we
can conclude that
\[
\aligned
\Vert\delta\Vert_V^2 &= \dual{\delta}{\delta-\Pi\delta}{V}=
\dual{\varepsilon-\varepsilon_h}{\delta-\Pi\delta}{V}\\ 
&=\dual{\varepsilon}{\delta-\Pi\delta}{V}+\dual{\varepsilon_h}{\Pi\delta}{V}.
\endaligned
\]
The properties of $\Pi$ lead to 
\[
\dual{\varepsilon}{\delta-\Pi\delta}{V}=b(u-u_h,\delta-\Pi\delta)
=\lambda(Mu)(\delta-\Pi\delta).
\]
From the previous estimate it follows
\[
\aligned
\Vert\delta\Vert_V^2&=
\lambda(Mu)(\delta-\Pi\delta)+\dual{\varepsilon_h}{\Pi\delta}{V}\\
&\le(\Vert\lambda(Mu)\circ(1-\Pi)\Vert_{V'}
+\Vert\varepsilon_h\Vert_V\Vert\Pi\Vert)\Vert\delta\Vert_V.
\endaligned
\]
Finally 
\[
C_1\|e\|_U\le\sqrt{\|\varepsilon_h\|^2_V+(\|\lambda(Mu)\circ(1-\Pi)\|_{V'}+
\|\varepsilon_h\|_V\|\Pi\|)^2}+C_M\|\lambda u_0-\lambda_h u_{0,h}\|.
\]
For the efficiency, we can use the decomposition~\eqref{decomposition}
and~\eqref{eq:inf-sup-continuity}, so that we obtain easily
\[
\aligned
\|\varepsilon_h\|_V^2 &=\|\varepsilon\|_V^2-\|\delta\|_V^2\\
&\le\|Be\|_{V'}^2+C_M^2\|\lambda u_0-\lambda_hu_{0,h}\|^2\\
&\le C_B^2\|e\|_U^2+C_M^2\|\lambda u_0-\lambda_hu_{0,h}\|^2.
\endaligned
\]
\end{proof}

\begin{corollary}
Proposition~\ref{r&r} holds true for the primal and ultra weak formulation 
of the Laplace eigenvalue problem that we have discussed in the previous
section.
\end{corollary}
\begin{proof}
The hypotheses stated in Proposition~\ref{r&r} are classical in the setting of
the DPG approximation of the Laplace problem; see, for instance~\cite[Section
3.2]{loworderultraweak} and~\cite[Section 3.1]{cdg}.
\end{proof}

\begin{remark}
The occurrence of a nonlinear term like $\lambda u_0-\lambda_hu_{0,h}$ is typical
when translating a posteriori analysis from the source to the eigenvalue
problem. Although this term is not suited for the standard AFEM setting, it is
generally a higher order term. We will comment more on this fact in
Section~\ref{se:HOT}.
\label{re:HOT}
\end{remark}

\subsection{An alternative error estimator}

In the case of lowest order approximations, we now present an error estimator
which depends only on the jump terms and that will turn out to be equivalent
to the natural error estimator $\eta$. Therefore for this we use similar 
arguments as in \cite{Hellwig2019Adaptive} for the source problem. The proof
relies on special properties of the Crouzeix--Raviart spaces,
which are defined as follows: 
\[
\aligned
&CR^1(\Omega_h):=\{v\in P_1(\Omega_h)|
v\text{ is continuous at }\mathrm{mid}(E)\quad\forall E\in\mathcal{E}(\Omega_h)\}\\
&CR^1_0(\Omega_h):=\{v\in CR^1(\Omega_h)|
v(\mathrm{mid}(E))=0\quad\forall E\in\mathcal{E}(\partial\Omega_h)\},
\endaligned
\]
where $\mathcal{E}(\Omega_h)$ and $\mathcal{E}(\partial\Omega_h)$ denote the sets
of interior and boundary edges of the triangulation, respectively.

This lemma from \cite[Lemma 3.2]{Hellwig2019Adaptive} shows a general orthogonality relationship between the
Crouzeix--Raviart spaces and continuous $P_1$ spaces and is crucial for the
equivalent statements.
\begin{lemma}\label{error_control_residual}
Any $w_{CR}\in CR^1_0(\Omega_h)$ with the $L^2$ orthogonality
$\grad_h w_{CR}\perp\grad S^1_0(\Omega_h)$ satisfies 
\[
\aligned
\vertiii{w_{CR}}^2_{pw}&\approx\sum_{T\in\Omega_h}\vert T\vert^{1/2} 
\sum_{E\in\mathcal E(T)}\Vert[\grad_h w_{CR}]_E\cdot\nu_E\Vert^2_{L^2(E)}\\
&\approx \sum_{T \in\Omega_h}\vert T\vert^{1/2} 
\sum_{E\in \mathcal E(T)}\Vert[\grad_h w_{CR}]_E\cdot\tau_E\Vert^2_{L^2(E)},
\endaligned
\]
where $\grad_h$ denotes the broken gradient.

\end{lemma}

We recall the lowest order approximation of the primal formulation where we
use the standard notation $u_h$ for $u_{0,h}$ and $\hat\sigma_h$ for
$u_{1,h}$: find $\lambda_h\in\CO$ such that for $(u_h,\hat\sigma_h)\in U_h$ and
$\varepsilon_h\in V_h$ it holds
\begin{equation}
\label{primal_formulation}
\aligned
    &(\varepsilon_h,v_h)_V + (\grad u_h,\grad v_h)_{\Omega_h} - \langle\hat\sigma_h,v_h 
    \rangle_{\partial \Omega_h} = \lambda_h (u_h,v_h)_{\Omega_h} && \forall v_h \in V_h \\
    &(\grad z_h,\grad \varepsilon_h)_{\Omega_h} -\langle \hat t_h, \varepsilon_h 
    \rangle_{\partial \Omega_h}= 0 && \forall (z_h,\hat t_h) \in U_h,
\endaligned
\end{equation}
where $U_h$ and $V_h$ are defined as follows
\[
\aligned
&U_{h,0}=S^1_0(\Omega_h)\\
&U_{h,1}=P_0(\partial\Omega_h)\cap U_1\\
&U_h = U_{h,0} \times U_{h,1}\\
&V_h=P_1(\Omega_h).
\endaligned
\]

The estimator we are looking at has been introduced
in~\cite{Hellwig2019Adaptive} for the source problem.
We define the following alternative error estimator $\bar \eta$ as follows
\begin{equation}
\bar{\eta}^2=\sum_{T \in \Omega_h} \vert T \vert^{1/2} 
\sum_{E\in \mathcal E (T)} \Vert [\grad_h \varepsilon_h]_E \Vert^2_{L^2(E)}.
\label{def_alternativ_errorestimator}
\end{equation}
The equivalence between the alternative error estimator and the one discussed
in the previous section is stated in the following theorem
and uses orthogonality arguments, which only hold in the lowest order case.
\begin{theorem}
Let $(u_h, \sigma_h) \in U_h$ and $\varepsilon_h \in V_h$ be the solution of 
the discrete primal problem \eqref{primal_formulation}.
Then the two error estimators defined in~\eqref{eq:eta1}
and~\eqref{def_alternativ_errorestimator} are equivalent. 
\end{theorem}
\begin{proof}
We follow here the arguments of \cite[Theorem 4.1]{Hellwig2019Adaptive}.

For any $\hat s_0\in U_{h,1}$ we have
$0=-b(0,\hat s_0;\varepsilon_h)=\dual{\hat s_0}{\varepsilon_h}{\partial\Omega_h}$. It
follows that $\int_E [\varepsilon_h]_E=0$ for all $E\in\partial\Omega_h$ so
that $\varepsilon_h$ belongs to $CR^1_0(\Omega_h)$. If we choose $(w,0)\in
U_h$ as test function, then $(\grad w_C,\grad\varepsilon_h)=0$ for all
$w_C\in S^1_0(\Omega_h)$ and the second property follows.

So $\varepsilon_h$ satisfies the assumptions of Lemma \ref{error_control_residual} and
we can conclude from~\eqref{def_alternativ_errorestimator} 
\[
\vertiii{\varepsilon_h}^2_{pw}\approx
\sum_{T \in\Omega_h}\|T\|^{1/2} 
\sum_{E\in\mathcal E(T)}\|[\grad_h\varepsilon_h]_E\|^2_{L^2(E)} = \bar \eta^2.
\]
Now from the discrete Friedrichs inequality for
Crouzeix--Raviart spaces~\cite{Brenner_2008} it follows
\[
\aligned
\bar\eta^2\approx\vertiii{\varepsilon_h}_{pw}^2&\le\|\varepsilon_h\|_{L^2(\Omega_h)}^2
+\vertiii{\varepsilon_h}_{pw}^2=\eta^2\\
&\lesssim\vertiii{\varepsilon_h}_{pw}^2\approx\bar\eta^2.
\endaligned
\]

\end{proof}

We observe, in particular, that the alternative error estimator $\bar{\eta}$
inherits the global efficiency and reliability from $\eta$.

To conclude this section, we remark that the alternative error estimator can be
defined for the ultra weak formulation as well. The formulation we are
considering is: find $\lambda_h\in\CO$ and
$(u_h,\bfsigma_h,\hat\sigma_{n,h},\hat u_h)\in U_h$ with $u_h\ne0$ such that for
$\varepsilon_h=(v_h,\bftau_h)\in V_h$ it holds
\begin{equation}
\left\{
\aligned
&(\varepsilon_h,(v,\bftau))_V+b(u_h,\bfsigma_h,\hat\sigma_{n,h},\hat u_h;v,\bftau)
=\lambda_h m(u_h,v)
&&\forall (v,\bftau)\in V_h\\
&b(w,r,s,t;\varepsilon_h)=0&&\forall (w,r,s,t)\in U_h,
\endaligned
\right.
\label{ultra_2}
\end{equation}
where the discrete spaces are given by
\[
U_h=
P_0(\Omega_h)
\times
P_0(\Omega_h; \mathbb R^2) 
\times S^1_0(\partial \Omega_h) \times P_0(\partial \Omega_h)
\]
and
\[
V_h=
P_1(\Omega_h)
\times
RT_0(\Omega_h).
\]

We can now define the alternative error estimator for the ultra weak
formulation analogously as we have done for the primal formulation as follows
\begin{equation}\label{eq:etatilde}
    \aligned
        \tilde{\eta}^2  :={}& \sum_{T \in \Omega_h} \vert T \vert^{1/2} 
        \sum_{E\in \mathcal E (T)} \Vert [\grad_h v_h]_E \Vert^2_{L^2(E)} \\
        ={}&\sum_{T \in \Omega_h} \vert T \vert^{1/2} 
        \sum_{E\in \mathcal E (T)} \Vert [\bftau_h]_E \Vert^2_{L^2(E)}.
    \endaligned
\end{equation}

Also in this case the alternative error estimator is equivalent to the natural
one with similar arguments as in~\cite[Theorem 4.4]{Hellwig2019Adaptive}.

\begin{theorem}
Let $\varepsilon_h = (v_h, \bftau_h) \in V_h$ and
$(u_h,\bfsigma_h,\hat\sigma_{n,h},\hat u_h) \in U_h$ be the solution of the
discrete ultra weak problem~\eqref{ultra_2}. Then the two error estimators for
the ultra weak formulation defined in~\eqref{eq:eta1} and~\eqref{eq:etatilde}
are equivalent. 
\end{theorem}
\begin{proof}
Similarly to the primal case, we first check that the assumptions of 
Lemma~\ref{error_control_residual} are fulfilled.

We can choose as test functions $(\tilde{w},\tilde{r}, 0, \tilde{s})\in U_h$. 
It follows from~\eqref{ultra_2} that $v_h$ belongs to $CR^1_0(\Omega_h)$, that
$\div \bftau_h=0$, and that $\bftau_h+\grad_h v=0$. Now we choose as test function
$(0,0,\tilde{s},0)$ with $\tilde{s}\in S^1_0(\Omega_h)$.
After integration by parts it follows
\[
\aligned
\dual{\bftau\cdot n}{\tilde{s}}{\partial \Omega_h}
&=(\div_h\bftau,\tilde{s})_{L^2(\Omega)}+(\bftau,\grad_h\tilde{s})_{L^2(\Omega)}\\
&=-(\grad_hv,\grad_h\tilde{s})_{L^2(\Omega)}=0,
\endaligned
\]
where $\div_h$ denotes the broken divergence operator.

From the identity $\bftau_h = -\nabla_h v_h$ just shown it follows

\begin{align*}
    \Vert v_h \Vert_{H^1(\Omega_h)}^2 &\leq \Vert \bftau_h\Vert_{H(\div,\Omega_h)}^2 + 
    \Vert v_h \Vert_{H^1(\Omega_h)}^2  = \Vert (\bftau_h,v_h)\Vert_V^2 \\
    & = \Vert \bftau_h \Vert_{L^2(\Omega)}^2 +\Vert \div \bftau_h \Vert^2_{L^2(\Omega)}
    + \Vert v_h \Vert_{H^1(\Omega_h)}^2 \\
    &= \Vert \grad v_h \Vert_{L^2(\Omega)}^2 +  \Vert v_h \Vert_{H^1(\Omega_h)}^2 
    \leq 2 \Vert v_h \Vert_{H^1(\Omega_h)}^2 .
\end{align*}
Now we can use Lemma~\ref{error_control_residual} together with the discrete
Friedrichs inequality to get the result
\[
\tilde \eta^2 \approx \vertiii{v_h}^2 \leq \Vert v_h \Vert_{H^1(\Omega_h)}^2
\leq \Vert (\bftau_h,v_h)  \Vert_V^2 = \eta^2 \leq 2 \Vert v_h \Vert_{H^1(\Omega_h)}^2
\lesssim \vertiii{v_h}^2 \approx \tilde \eta ^2.
\]
\end{proof}
Hence, also in this case, $\tilde{\eta}$ inherits the global efficiency and
reliability from $\eta$.

\subsection{The nonlinear term $\|\lambda u_0-\lambda_h u_{0,h}\|$}
\label{se:HOT}

We show in this section how to compare the nonlinear term
$\|\lambda u_0-\lambda_h u_{0,h}\|$ to the error $\|u-u_h\|_U$.
We are going to use the following identity
\begin{equation}
\lambda u_0-\lambda_hu_{0,h}=\lambda(u_0-u_{0,h})+u_{0,h}(\lambda-\lambda_h).
\label{eq:+-}
\end{equation}
A standard way to show that this is a higher order term is to observe that
$\|u_0-u_{0,h}\|$ is converging in the $L^2(\Omega)$ norm faster than the
error in the energy norm and that $|\lambda-\lambda_h|$ is converging with
double order.

In the DPG approximations that we are discussing, we can see that this
property is valid for $|\lambda-\lambda_h|$, while particular attention has to
be paid for the term $\|u_0-u_{0,h}\|$. Indeed, the a priori estimates for the
primal formulation recalled in~\eqref{eq:estprimalfull}
and~\eqref{eq:estprimal} guarantee that $\|u_0-u_{0,h}\|_{L^2(\Omega)}$ is of
higher order with respect to $\|u-u_h\|_U$, while this is not always the case
for the ultra weak formulation. When the augmented formulation is used,
Theorem~\ref{th:UW+} guarantees that we have the desired property. We
summarize the obtained results in the following statement.

\begin{proposition}
\label{pr:HOT}
In the case of the primal DPG formulation presented in Section~\ref{se:primal}
and of the augmented  ultra weak DPG formulation presented in
Section~\ref{se:UW}, the nonlinear term
$\|\lambda u_0-\lambda_h u_{0,h}\|_{L^2(\Omega)}$ is of order
$O(h^{\tau+\upsilon})$ when $h$ goes to zero, with respect to the energy norm
error $\|u-u_h\|_U$ which is of order $O(h^\tau)$, where $\tau$ and $\upsilon$
are defined in Theorems~\ref{th:th1} and~\ref{th:th2},in
Proposition~\ref{pr:pr1}, and in Theorem~\ref{th:UW+}.
\end{proposition}

\begin{proof}
The result follows from the identity~\eqref{eq:+-} and the a priori estimates
presented in~\eqref{eq:estprimalfull}, \eqref{eq:estprimal} and in
Theorem~\ref{th:UW+}.
\end{proof}

\section{Numerical results}
\label{se:num}

This section presents selected numerical experiments in two dimensions for the
primal and ultra weak formulations. We implemented the lowest order method
based on~\cite{AFEM} and~\cite{Hellwig2019Adaptive}. For the adaptive mesh
refinement we used the classical AFEM algorithm~\cite{AoA} with D\"orfler
marking and bulk parameter $\theta =0.5$. 
For the higher order formulations we used NETGEN/NGSOLVE \cite{NETGEN,NGSOLVE}.
In case of the ultraweak formulation we used the discretization 
from \eqref{ultradiskret}.
In Figure~\ref{mesh_figure} we show some of our mesh obtained after four
refinements. As expected, the meshes are strongly refined where singular
solutions are expected.
\begin{figure}[ht]
    \begin{subfigure}[b]{0.45\textwidth}
        \subcaption{Slit domain adaptive mesh for primal DPG}
        \includegraphics[scale = 0.41]{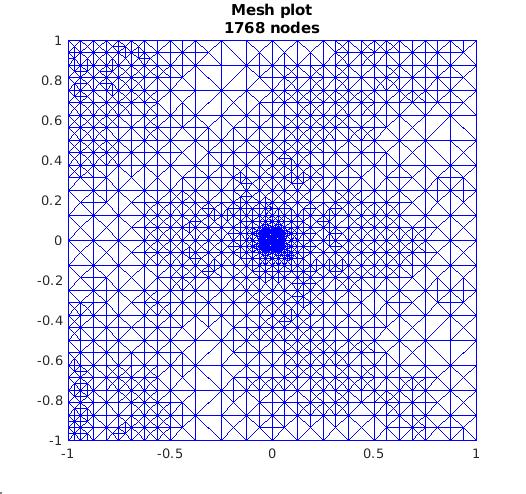}
    \end{subfigure}
    \hfill
    \begin{subfigure}[b]{0.45\textwidth}
        \subcaption{L-shaped domain adaptive mesh for primal DPG}
        \includegraphics[scale = 0.41]{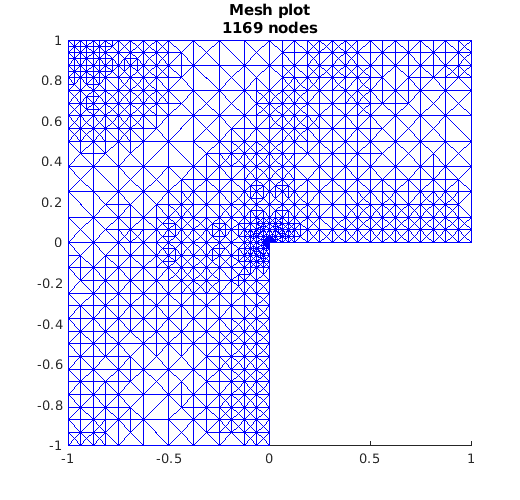}
    \end{subfigure}
    \label{adaptive_meshes}
    \caption{Adaptive meshes after 4 Iterations}
    \label{mesh_figure}
\end{figure}
\subsection{Numerical results on the square domain}
Our first domain is the convex square domain, which is defined
as $\Omega = [0,1]^2$. For this domain the exact solution is well
known. Here we compare our solution with the first eigenvalue 
$\lambda_1 = 2\pi^2$. In Figure \ref{SquareHO} we can clearly 
see that the methods show the expected convergence rates.

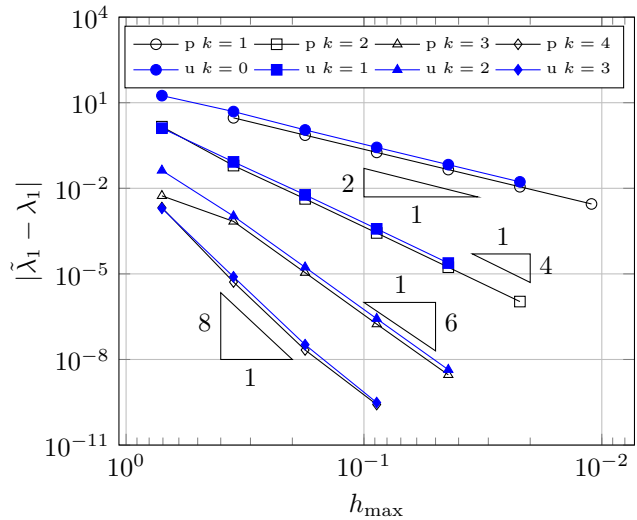
\begin{figure}[ht]
\begin{tikzpicture}
    \begin{loglogaxis}[
        xlabel={$h_{\mathrm{max}}$},
        ylabel={$|\tilde\lambda_1 -\lambda_1|$},
        grid=major,
        legend style={font=\tiny,legend columns=4},
        ymax=10e3,
        ymin=10e-12,
        x dir= reverse]
        \addplot[color=black, mark=o] table [x=hmax,y=err]
        {data/primal_square_k1.txt};
        \addlegendentry{p $k=1$}

        \addplot[color=black, mark=square] table [x=hmax,y=err]
        {data/primal_square_k2.txt};
        \addlegendentry{p $k=2$}
        
        \addplot[color=black, mark=triangle] table [x=hmax,y=err]
        {data/primal_square_k3.txt};
        \addlegendentry{p $k=3$}
       
        \addplot[color=black, mark=diamond] table [x=hmax,y=err]
        {data/primal_square_k4.txt};
        \addlegendentry{p $k=4$}
        
        \addplot[color=blue, mark=*] table [x=hmax,y=err]
        {data/ultra_square_k0.txt};
        \addlegendentry{u $k=0$}

        \addplot[color=blue, mark=square*] table [x=hmax,y=err]
        {data/ultra_square_k1.txt};
        \addlegendentry{u $k=1$}
        
        \addplot[color=blue, mark=triangle*] table [x=hmax,y=err]
        {data/ultra_square_k2.txt};
        \addlegendentry{u $k=2$}
        
        \addplot[color=blue, mark=diamond*] table [x=hmax,y=err]
        {data/ultra_square_k3.txt};
        \addlegendentry{u $k=3$}
        
        \draw (axis cs:0.1,0.05) -- (axis cs:0.1,0.005) -- (axis cs:0.0333,0.005) -- cycle;
        \node at (axis cs:0.1,0.015)[left] {$2$};
        \node at (axis cs:0.06,0.005)[below] {$1$};
        
        \draw (axis cs:0.02,0.00005) -- (axis cs:0.02,0.000005) -- (axis cs:0.035,0.00005) -- cycle;
        \node at (axis cs:0.02,0.00002 )[right]{$4$};
        \node at (axis cs:0.026,0.00005)[above]{$1$};

        \draw (axis cs:0.1,0.000001) -- (axis cs:0.05,0.000001) -- (axis cs:0.05,0.00000002) -- cycle;
        \node at (axis cs:0.07,0.000001)[above]{$1$};
        \node at (axis cs:0.05,0.0000002)[right]{$6$};

        \draw (axis cs:0.4,0.0000022) -- (axis cs:0.4,0.00000001) -- (axis cs:0.2,0.00000001) -- cycle;
        \node at (axis cs:0.4,0.0000002)[left]{$8$};
        \node at (axis cs:0.3,0.00000001)[below]{$1$};
    \end{loglogaxis}
\end{tikzpicture}
\caption{Convergence rates on the square domain, uniform refinement (p =
primal, u = ultra weak)}\label{SquareHO}
\end{figure}

\subsection{Numerical results on the L-shaped domain}
On the non-convex L-shaped domain
$\Omega = (-1,1)^2 \setminus \big([0,1) \times (-1,0]\big)$ we used
the reference values  $$\lambda_1 = 9.639723844871536\ \text{and}\ \lambda_5 = 31.91263.$$
These two eigenvalues correspond to singular eigenspaces that belong to
$H^{1+s}(\Omega)$ with $s<1/2$.
For this reason we expect a convergence rate of $2/3$ in terms of the number of
degrees of freedom when using a uniform mesh refinement. This rate of
convergence is clearly detected for $\lambda_1$ in
Figure~\ref{first_eigenvalue}, while in the case of $\lambda_5$ a
pre-asymptotic convergence is detected: the rate of convergence is
approximately $0.78$ in the last iteration of Figure~\ref{fifth_eigenvalue},
but it can be seen that the rate is degenerating as the mesh is refined.
For bulk parameter $\theta =0.5$ all adaptive methods show optimal convergence
rates.

For the higher order methods figure \ref{lshapeHO} shows a similar behaviour 
as in the lowest order case. Moreover, the convergence rate can be restored 
in the adaptive case (Figure \ref{lshapeadapHO}).

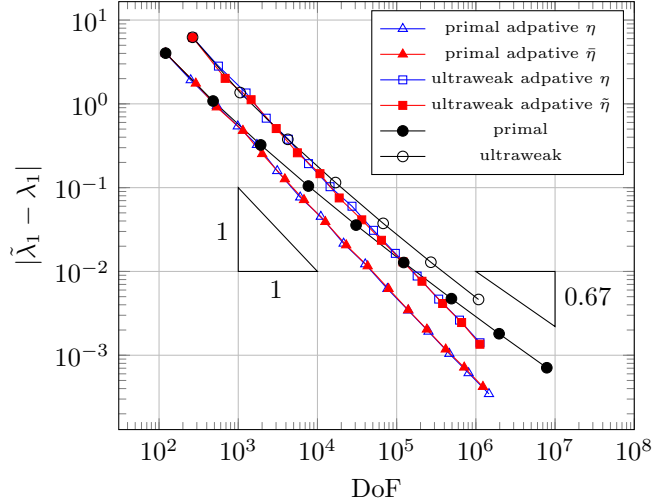
\begin{figure}[ht]
\begin{tikzpicture}
    \begin{loglogaxis}[
        xlabel={DoF},
        ylabel={$|\tilde\lambda_1- \lambda_1| $},
        grid=major,
        legend style={font=\tiny},
        xmax=10e7]

        \addplot[color=blue,mark=triangle] table [x=dof, y=abserror1, color =blue]
        {data/primal_Lshape_t=0.5_epsilon.dat};
        \addlegendentry{primal adpative $\eta$}

        \addplot[color=red,mark=triangle*] table [x=dof, y=abserror1]
        {data/primal_Lshape_t=0.5_reduced.dat};
        \addlegendentry{primal adpative $\bar \eta$}

        \addplot[color=blue,mark=square,mark size=1.5pt] table [x=dof, y=abserror1]
        {data/ultraweak_Lshape_t=0.5_epsilon.dat};
        \addlegendentry{ultraweak adpative $\eta$}

        \addplot[color=red,mark=square*,mark size=1.5pt ] table [x=dof, y=abserror1]
        {data/ultraweak_Lshape_t=0.5_reduced.dat};
        \addlegendentry{ultraweak adpative $\tilde \eta$}

        \addplot[color=black,mark=*] table [x=dof, y=abserror1]
        {data/primal_Lshape_ref.dat};
        \addlegendentry{primal}

        \addplot[color=black,mark=o] table [x=dof, y=abserror1]
        {data/ultraweak_Lshape_ref.dat};
        \addlegendentry{ultraweak}

        \draw (axis cs:10000,1e-2)--(axis cs:1000 ,1e-2)--(axis cs:1000 ,1e-1)--cycle;
        \node at (axis cs:3000,1e-2) [below] {$1$};
        \node at (axis cs:1000,0.03) [left] {$1$};

        \draw (axis cs:10000000,0.01)--(axis cs: 10000000,0.0022)--(axis cs: 1000000,0.01)--cycle;
        \node at (axis cs:10000000,0.005) [right] {$0.67$};
    \end{loglogaxis}
\end{tikzpicture}
\caption{Convergence rates for the L-shaped domain, first eigenvalue}\label{first_eigenvalue}
\end{figure}

\begin{figure}[ht]
    \begin{tikzpicture}
        \begin{loglogaxis}[
        xlabel={DoF},
        ylabel={$|\tilde\lambda_5 - \lambda_5| $},
        grid=major,
        legend style={font=\tiny},
        xmax=30000000,
        ymax=10e1,
        ]

        \addplot[color=blue,mark=triangle] table [x=dof, y=abserror1, color =blue]
        {data/primal_Lshape_t=0.5_epsilon_EV5.dat};
        \addlegendentry{primal adpative $\eta$}

        \addplot[color=red,mark=triangle*] table [x=dof, y=abserror1]
        {data/primal_Lshape_t=0.5_reduced_EV5.dat};
        \addlegendentry{primal adpative $\bar \eta$}

        \addplot[color=blue,mark=square] table [x=dof, y=abserror1]
        {data/ultraweak_Lshape_t=0.5_epsilon_EV5.dat};
        \addlegendentry{ultraweak adpative $\eta$}

        \addplot[color=red,mark=square*] table [x=dof, y=abserror1]
        {data/ultraweak_Lshape_t=0.5_reduced_EV5.dat};
        \addlegendentry{ultraweak adpative $\tilde \eta$}

        \addplot[color=black,mark=*] table [x=dof, y=abserror1]
        {data/primal_Lshape_ref_EV5.dat};
        \addlegendentry{primal}

        \addplot[color=black,mark=o] table [x=dof, y=abserror1]
        {data/ultraweak_Lshape_ref_EV5.dat};
        \addlegendentry{ultraweak}
        
        \draw (axis cs:10000,1e-2)--(axis cs:1000 ,1e-2)--(axis cs:1000 ,1e-1)--cycle;
        \node at (axis cs:3000,1e-2) [below] {$1$};
        \node at (axis cs:1000,0.03) [left] {$1$};

        \draw (axis cs:1e5,1e-1) -- (axis cs:1e6,0.02173) -- (axis cs:1e6,1e-1)-- cycle;
        \node at (axis cs:300000,1e-1)[above]{$1$};
        \node at (axis cs:1e6,0.07)[right]{$\approx 0.78$};
        \end{loglogaxis}
    \end{tikzpicture}
    
    \caption{Convergence rates for the L-shaped domain, fifth eigenvalue}\label{fifth_eigenvalue}
\end{figure}
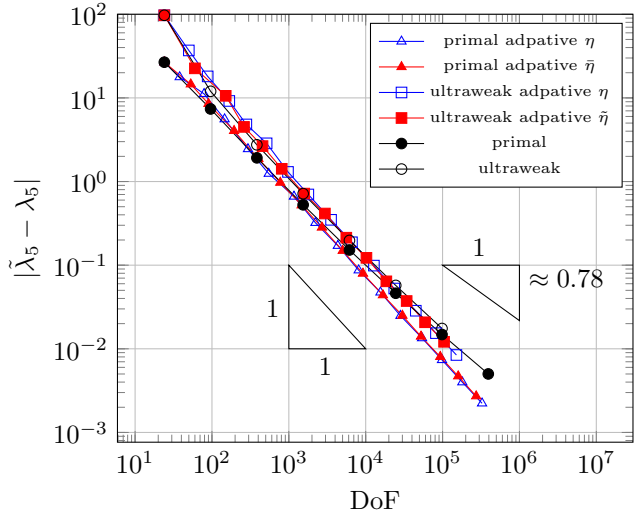

\begin{figure}[ht]
\begin{tikzpicture}
    \begin{loglogaxis}[
        xlabel={$h_{\mathrm{max}}$},
        ylabel={$ |\tilde\lambda_1 -\lambda_1|$},
        grid=major,
        legend style={font=\tiny,legend columns=4},
        ymax=150,
        ymin=10e-5,
        x dir= reverse]
        \addplot[color=black, mark=o] table [x=hmax,y=err]
        {data/primal_lshape_k1.txt};
        \addlegendentry{p $k=1$}

        \addplot[color=black, mark=square] table [x=hmax,y=err]
        {data/primal_lshape_k2.txt};
        \addlegendentry{p $k=2$}
        
        \addplot[color=black, mark=triangle] table [x=hmax,y=err]
        {data/primal_lshape_k3.txt};
        \addlegendentry{p $k=3$}
       
        \addplot[color=black, mark=diamond] table [x=hmax,y=err]
        {data/primal_lshape_k4.txt};
        \addlegendentry{p $k=4$}
        
        \addplot[color=blue, mark=*] table [x=hmax,y=err]
        {data/ultra_lshape_k0.txt};
        \addlegendentry{u $k=0$}

        \addplot[color=blue, mark=square*] table [x=hmax,y=err]
        {data/ultra_lshape_k1.txt};
        \addlegendentry{u $k=1$}
        
        \addplot[color=blue, mark=triangle*] table [x=hmax,y=err]
        {data/ultra_lshape_k2.txt};
        \addlegendentry{u $k=2$}
        
        \addplot[color=blue, mark=diamond*] table [x=hmax,y=err]
        {data/ultra_lshape_k3.txt};
        \addlegendentry{u $k=3$}
        
        \draw (axis cs:0.3,0.01) -- (axis cs:0.3,0.0022) -- (axis cs:0.1,0.0022) -- cycle;
        \node at (axis cs:0.3,0.005)[left] {$\approx1.33$};
        \node at (axis cs:0.2,0.002)[below] {$1$};
        
    \end{loglogaxis}
\end{tikzpicture}
\caption{Convergence rates for higher order on L-shaped domain (p = primal, u
= ultraweak)}\label{lshapeHO}
\end{figure}
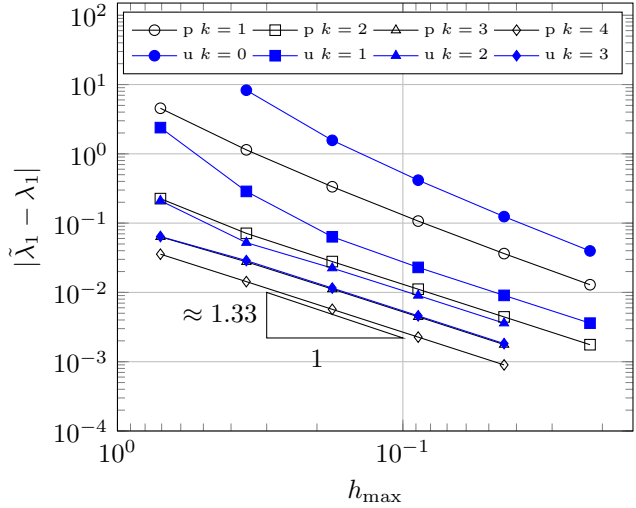

\begin{figure}[ht]
\begin{tikzpicture}
    \begin{loglogaxis}[
        xlabel={DoF},
        ylabel={$ |\tilde\lambda_1 - \lambda_1|  $},
        grid=major,
        legend style={font=\tiny,legend columns=4},
        xmax=10e5,
        ymax=10000,
        ymin=10e-12]
        \addplot[color=black, mark=o] table [x=dof,y=err]
        {data/adap_primal_lshape_k1.txt};
        \addlegendentry{p $k=1$}
        \addplot[color=blue, mark=square] table [x=dof,y=err]
        {data/adap_primal_lshape_k2.txt};
        \addlegendentry{p $k=2$}
        
        \addplot[color=red, mark=triangle] table [x=dof,y=err]
        {data/adap_primal_lshape_k3.txt};
        \addlegendentry{p $k=3$}

        \addplot[color=green, mark=diamond] table [x=dof,y=err]
        {data/adap_primal_lshape_k4.txt};
        \addlegendentry{p $k=4$}
        
        \addplot[color=black, mark=*] table [x=dof,y=err]
        {data/adap_ultra_lshape_k0.txt};
        \addlegendentry{u $k=0$}
        
        \addplot[color=blue, mark=square*] table [x=dof,y=err]
        {data/adap_ultra_lshape_k1.txt};
        \addlegendentry{u $k=1$}
        
        \addplot[color=red, mark=triangle*] table [x=dof,y=err]
        {data/adap_ultra_lshape_k2.txt};
        \addlegendentry{u $k=2$}
        
        \addplot[color=green, mark=diamond*] table [x=dof,y=err]
        {data/adap_ultra_lshape_k3.txt};
        \addlegendentry{u $k=3$}
        
        \draw (axis cs:100000,1) -- (axis cs:10000,1) -- (axis cs:100000,0.1) -- cycle;
        \node at (axis cs:33000,1)[above] {$1$};
        \node at (axis cs:100000,0.3)[right] {$1$};

        \draw (axis cs:200000,0.001) -- (axis cs:200000,0.0001) -- (axis cs:60000,0.001) -- cycle;
        \node at (axis cs:200000,0.0003)[right]{$2$};
        \node at (axis cs:100000,0.00075)[above]{$1$};

        \draw (axis cs:10000,0.00001) -- (axis cs:4000,0.00001) -- (axis cs:4000,0.0002) -- cycle;
        \node at (axis cs:4000,0.00006)[left]{$3$};
        \node at (axis cs:6500,0.00001)[below]{$1$};

        \draw (axis cs:10000,0.000001) -- (axis cs:10000,0.00000001) -- (axis cs: 20000, 0.00000001) -- cycle;
        \node at (axis cs:10000,0.00000008)[left]{$4$};
        \node at (axis cs:14000,0.000000001)[above]{$1$};
        
    \end{loglogaxis}
\end{tikzpicture}
\caption{Adaptive convergence rates for the L-shaped domain for higher order(p
= primal, u = ultraweak)}\label{lshapeadapHO}
\end{figure}
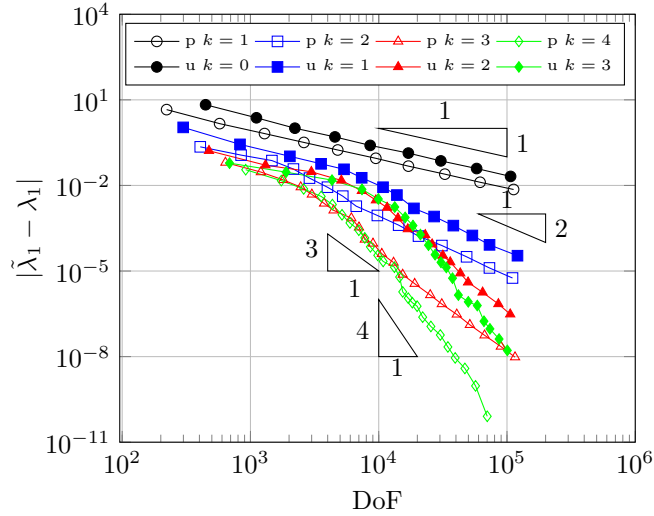
\subsection{Numerical results on the slit domain}
On the non-convex slit domain
$\Omega = (-1,1)^2 \setminus \big([0,1) \times \{0\}\big)$ we 
used $\lambda_1 = 8.371329711$ as reference solution.
For the first eigenvalue we expect a convergence 
rate of $1/2$ when the mesh is refined uniformly. Like for the L-shape domain, the
adaptive methods can recover the optimal convergence rate.

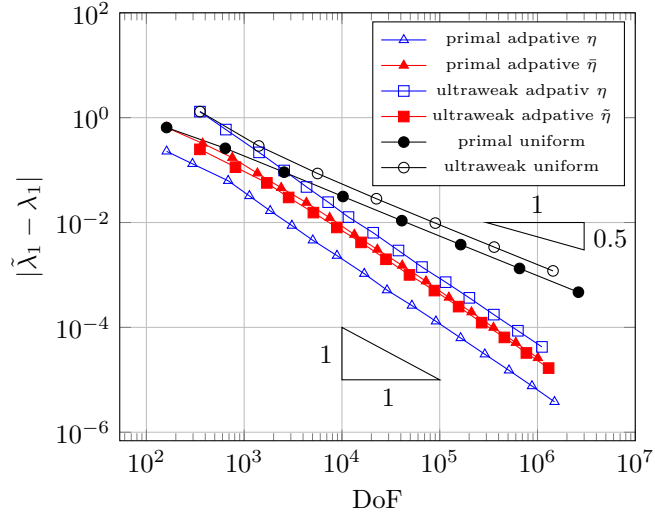
\begin{figure}[ht]
\begin{tikzpicture}
    \begin{loglogaxis}[
        xlabel={DoF},
        ylabel={$|\tilde\lambda_1- \lambda_1 |$},
        grid=major,
        legend style={font=\tiny},
        xmax=10e6,
        ymax=10e1,
        ]
        
        \addplot[color=blue, mark=triangle] table [x=dof, y=abserror1]
        {data/primal_SlitDb_t=0.5_epsilon.dat};
        \addlegendentry{primal adpative $\eta$}
        \addplot[color=red, mark=triangle*] table [x=dof, y=abserror1]
        {data/primal_SlitDb_t=0.5_reduced.dat};
        \addlegendentry{primal adpative $\bar \eta$}
        \addplot[color=blue, mark=square] table [x=dof, y=abserror1]
        {data/ultraweak_SlitDb_t=0.5_epsilon.dat};
        \addlegendentry{ultraweak adpativ $\eta$}
        \addplot[color=red, mark=square*] table [x=dof, y=abserror1]
        {data/ultraweak_SlitDb_t=0.5_reduced.dat};
        \addlegendentry{ultraweak adpative $\tilde{\eta}$}
        \addplot[color=black, mark=*] table [x=dof, y=abserror1]
        {data/primal_SlitDb_ref.dat};
        \addlegendentry{primal uniform}
        \addplot[color=black, mark=o] table [x=dof, y=abserror1]
        {data/ultraweak_SlitDb_ref.dat};
        \addlegendentry{ultraweak uniform}
        
    \draw (axis cs:100000,1e-5)--(axis cs:10000 ,1e-5)--(axis cs:10000 ,1e-4)--cycle;
        \node at (axis cs:30000,1e-5) [below] {$1$};
        \node at (axis cs:10000,0.00003) [left] {$1$};

    \draw (axis cs:300000,1e-2)--(axis cs:3000000,1e-2)--(axis cs:3000000,0.003)--cycle;
        \node at (axis cs: 1e6,1e-2) [above]{$1$};
        \node at (axis cs: 3000000,0.005)[right]{$0.5$};
    \end{loglogaxis}
\end{tikzpicture}
\caption{Convergence rates for the slit domain}
\end{figure}
\subsection{Higher order term}
In our next numerical simulation, we check in the lowest order case the
statement of Proposition~\ref{pr:HOT} about the higher order term appearing in our
a posteriori estimates.
In order to calculate the higher order term, we computed a reference solution
on a fine mesh with about a million of degrees of freedom.
In Figure~\ref{fig:HO} we can appreciate that the term
$\Vert \lambda u_0 -\lambda_h u_{0,h}\Vert$ is actually of higher order and
converges twice as fast as the eigenfunction in both cases that we have
considered. The rate is actually stabilizing about the value of $1$, even if
on the slit domain the convergence is faster in the pre-asymptotic regime.

\begin{figure}[ht]
\begin{tikzpicture}
    \begin{loglogaxis}[
        xlabel={DoF},
        ylabel={$\Vert \tilde\lambda_1 \tilde u_0- \lambda_1 u_0\Vert_U $},
        grid=major,
        legend style={font=\tiny},
        xmax=10e5,
        ymin=10e-5]
        \addplot[color=blue, mark=*] table [x=dof,y=HO]
        {data/primal_Lshape_t=0.5_recuded_HO.dat};
        \addlegendentry{primal L-shaped}
        \addplot[color= red,mark= +] table [x=dof,y=HO]
        {data/primal_SlitDb_t=0.5_recuded_HO.dat};
        \addlegendentry{primal Slit}

        \draw (axis cs:30000,0.0005) -- (axis cs:300000,0.0005) -- (axis cs:30000,0.005) -- cycle;
        \node at (axis cs:90000,0.0005)[below] {$1$};
        \node at (axis cs:30000,0.0015)[left]  {$1$};
    \end{loglogaxis}
\end{tikzpicture}
\caption{Convergence higher order term}
\label{fig:HO}
\end{figure}
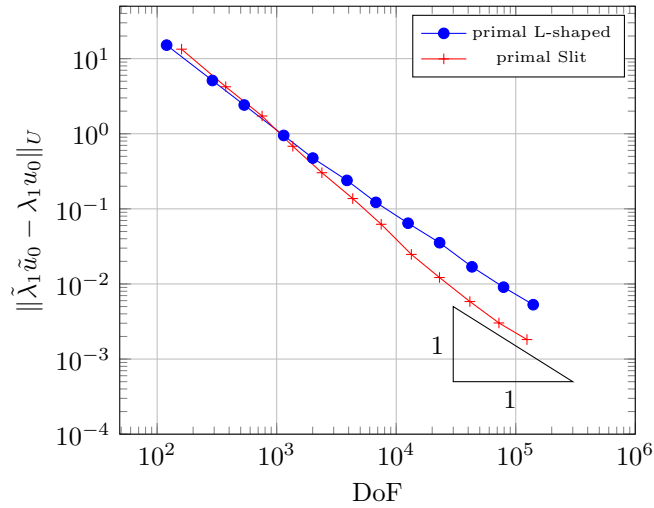
\subsection{Efficiency-ratio}
Our last numerical test concerns the efficiency-ratio, defined as
$ \frac{\eta}{\Vert \tilde u - u\Vert_U}$. We used the same reference solution
that we computed for the higher order term calculation.
In table \ref{tab:eff} we can see that for the L-shaped domain we have an
efficiency-ratio between 9 and 10. For the slit domain we have a ratio
between 14 to 19.

\begin{table}
\sisetup{round-mode=places, round-precision=3}
\begin{filecontents*}{data/primal_Lshaperatio.csv}
dof_lshape,ratio_lshape,dof_slit,ratio_slit
120,    7.75870861187147,160,8.40636428792684
290,    7.55650375498635,375,8.21776434765294
535,7.71827770882474,755,8.85162608196371
1145,8.91395197419766,1365,10.1905181071799
2005,9.0644963714991,2390,11.4354552788041
3875,9.29017369608367,4330,12.7566068272323
6765,9.71358499602191,7490,14.1605057471264
12565,9.84282473387238,13390,16.8835074027316
23075,9.47094172531035,23020,17.7130616025938
43055,10.1996253237093,41345,18.2749186452935
79080,9.9984598008009,72360,17.5898863372767
140000,9.25825140449438,124155,14.3411719352352
\end{filecontents*}
\csvreader[tabular=|l|l||l|l|,
            table head= \hline DoF & Efficiency-ratio (L-shaped) & DoF & Efficiency-ratio (Slit)\\\hline,
            late after line= \\\hline]
            {data/primal_Lshaperatio.csv}{dof_lshape=\doflshape, ratio_lshape=\ratiolshape, dof_slit=\dofslit ,ratio_slit=\ratioslit}
            { \doflshape & \num{\ratiolshape} & \dofslit & \num{\ratioslit}}
\caption{Efficiency ratio for primal error estimator $\eta$}
\label{tab:eff}
\end{table}

\bibliographystyle{amsplain}
\bibliography{ref}

\end{document}